\documentclass[12pt]{amsart}
\usepackage{amsmath,amssymb,amsbsy,amsfonts,latexsym,amsopn,amstext,cite,
                                               amsxtra,euscript,amscd,bm}
\usepackage{url}
\usepackage{mathrsfs}

\usepackage{color}
\usepackage[colorlinks,linkcolor=blue,anchorcolor=blue,citecolor=blue,backref=page]{hyperref}
\usepackage{color}
\usepackage{graphics,epsfig}
\usepackage{graphicx}
\usepackage{float}
\usepackage{epstopdf}
\hypersetup{breaklinks=true}

\usepackage[np]{numprint}
\npdecimalsign{\ensuremath{.}}

\def\Th1{\varTheta}

\usepackage{bibentry}

\usepackage[english]{babel}
\usepackage{mathtools}
\usepackage{todonotes}
\usepackage{url}
\usepackage[colorlinks,linkcolor=blue,anchorcolor=blue,citecolor=blue,backref=page]{hyperref}

\usepackage[norefs,nocites]{refcheck}

\begin{document}

\newtheorem{theorem}{Theorem}
\newtheorem{lemma}[theorem]{Lemma}
\newtheorem{claim}[theorem]{Claim}
\newtheorem{cor}[theorem]{Corollary}
\newtheorem{conj}[theorem]{Conjecture}
\newtheorem{prop}[theorem]{Proposition}
\newtheorem{definition}[theorem]{Definition}
\newtheorem{question}[theorem]{Question}
\newtheorem{example}[theorem]{Example}
\newcommand{\hh}{{{\mathrm h}}}
\newtheorem{remark}[theorem]{Remark}

\numberwithin{equation}{section}
\numberwithin{theorem}{section}
\numberwithin{table}{section}
\numberwithin{figure}{section}

\def\sssum{\mathop{\sum\!\sum\!\sum}}
\def\ssum{\mathop{\sum\ldots \sum}}
\def\iint{\mathop{\int\ldots \int}}

\newcommand{\diam}{\operatorname{diam}}

\def\squareforqed{\hbox{\rlap{$\sqcap$}$\sqcup$}}
\def\qed{\ifmmode\squareforqed\else{\unskip\nobreak\hfil
\penalty50\hskip1em \nobreak\hfil\squareforqed
\parfillskip=0pt\finalhyphendemerits=0\endgraf}\fi}

\newfont{\teneufm}{eufm10}
\newfont{\seveneufm}{eufm7}
\newfont{\fiveeufm}{eufm5}
%
%
\newfam\eufmfam
     \textfont\eufmfam=\teneufm
\scriptfont\eufmfam=\seveneufm
     \scriptscriptfont\eufmfam=\fiveeufm
%
%
\def\frak#1{{\fam\eufmfam\relax#1}}

\newcommand{\bflambda}{{\boldsymbol{\lambda}}}
\newcommand{\bfmu}{{\boldsymbol{\mu}}}
\newcommand{\bfxi}{{\boldsymbol{\eta}}}
\newcommand{\bfrho}{{\boldsymbol{\rho}}}

\def\eps{\varepsilon}

\def\fK{\mathfrak K}
\def\fT{\mathfrak{T}}
\def\fL{\mathfrak L}
\def\fR{\mathfrak R}

\def\fA{{\mathfrak A}}
\def\fB{{\mathfrak B}}
\def\fC{{\mathfrak C}}
\def\fM{{\mathfrak M}}
\def\fS{{\mathfrak  S}}
\def\fU{{\mathfrak U}}
\def\fW{{\mathfrak W}}

\def\T {\mathsf {T}}
\def\Tor{\mathsf{T}_d}
\def\Tore{\widetilde{\mathrm{T}}_{d} }

\def\sM {\mathsf {M}}

\def\ss{\mathsf {s}}

\def\Kmnd{\cK_d(m,n)}
\def\Kmnp{\cK_p(m,n)}
\def\Kmnq{\cK_q(m,n)}

\def \balpha{\bm{\alpha}}
\def \bbeta{\bm{\beta}}
\def \bgamma{\bm{\gamma}}
\def \bdelta{\bm{\delta}}
\def \bzeta{\bm{\zeta}}
\def \blambda{\bm{\lambda}}
\def \bchi{\bm{\chi}}
\def \bphi{\bm{\varphi}}
\def \bpsi{\bm{\psi}}
\def \bnu{\bm{\nu}}
\def \bomega{\bm{\omega}}

\def \bell{\bm{\ell}}

\def\eqref#1{(\ref{#1})}

\def\vec#1{\mathbf{#1}}

\newcommand{\abs}[1]{\left| #1 \right|}

\def\Zq{\mathbb{Z}_q}
\def\Zqx{\mathbb{Z}_q^*}
\def\Zd{\mathbb{Z}_d}
\def\Zdx{\mathbb{Z}_d^*}
\def\Zf{\mathbb{Z}_f}
\def\Zfx{\mathbb{Z}_f^*}
\def\Zp{\mathbb{Z}_p}
\def\Zpx{\mathbb{Z}_p^*}
\def\cM{\mathcal M}
\def\cE{\mathcal E}
\def\cH{\mathcal H}

\def\le{\leqslant}

\def\ge{\geqslant}

\def\sfB{\mathsf {B}}
\def\sfC{\mathsf {C}}
\def\L{\mathsf {L}}
\def\FF{\mathsf {F}}

\def\sE {\mathscr{E}}
\def\sS {\mathscr{S}}

\def\cA{{\mathcal A}}
\def\cB{{\mathcal B}}
\def\cC{{\mathcal C}}
\def\cD{{\mathcal D}}
\def\cE{{\mathcal E}}
\def\cF{{\mathcal F}}
\def\cG{{\mathcal G}}
\def\cH{{\mathcal H}}
\def\cI{{\mathcal I}}
\def\cJ{{\mathcal J}}
\def\cK{{\mathcal K}}
\def\cL{{\mathcal L}}
\def\cM{{\mathcal M}}
\def\cN{{\mathcal N}}
\def\cO{{\mathcal O}}
\def\cP{{\mathcal P}}
\def\cQ{{\mathcal Q}}
\def\cR{{\mathcal R}}
\def\cS{{\mathcal S}}
\def\cT{{\mathcal T}}
\def\cU{{\mathcal U}}
\def\cV{{\mathcal V}}
\def\cW{{\mathcal W}}
\def\cX{{\mathcal X}}
\def\cY{{\mathcal Y}}
\def\cZ{{\mathcal Z}}
\newcommand{\rmod}[1]{\: \mbox{mod} \: #1}

\def\cg{{\mathcal g}}

\def\vy{\mathbf y}
\def\vr{\mathbf r}
\def\vx{\mathbf x}
\def\va{\mathbf a}
\def\vb{\mathbf b}
\def\vc{\mathbf c}
\def\ve{\mathbf e}
\def\vh{\mathbf h}
\def\vk{\mathbf k}
\def\vm{\mathbf m}
\def\vz{\mathbf z}
\def\vu{\mathbf u}
\def\vv{\mathbf v}

\def\e{{\mathbf{\,e}}}
\def\ep{{\mathbf{\,e}}_p}
\def\eq{{\mathbf{\,e}}_q}

\def\Tr{{\mathrm{Tr}}}
\def\Nm{{\mathrm{Nm}}}

 \def\SS{{\mathbf{S}}}

\def\lcm{{\mathrm{lcm}}}

 \def\0{{\mathbf{0}}}

\def\({\left(}
\def\){\right)}
\def\l|{\left|}
\def\r|{\right|}
\def\fl#1{\left\lfloor#1\right\rfloor}
\def\rf#1{\left\lceil#1\right\rceil}
\def\fl#1{\left\lfloor#1\right\rfloor}
\def\ni#1{\left\lfloor#1\right\rceil}
\def\sumstar#1{\mathop{\sum\vphantom|^{\!\!*}\,}_{#1}}

\def\mand{\qquad \mbox{and} \qquad}

\def\tblue#1{\begin{color}{blue}{{#1}}\end{color}}




\hyphenation{re-pub-lished}

\mathsurround=1pt

\def\bfdefault{b}

\def \F{{\mathbb F}}
\def \K{{\mathbb K}}
\def \N{{\mathbb N}}
\def \Z{{\mathbb Z}}
\def \P{{\mathbb P}}
\def \Q{{\mathbb Q}}
\def \R{{\mathbb R}}
\def \C{{\mathbb C}}
\def\Fp{\F_p}
\def \fp{\Fp^*}

 \def \xbar{\overline x}

\title[Monochromatic Solutions Using At Least Three Colors]
{On Monochromatic Solutions of Linear Equations Using At Least Three Colors}

\author [L. P. Wijaya]{Laurence P. Wijaya}
\address{Department of Mathematics, University of Kentucky, 715 Patterson Office Tower, Lexington, KY 40506, USA}
\email{laurence.wijaya@uky.edu}

\begin{abstract}  We study the number of monochromatic solutions to linear equation in $\{1,\dots,n\}$ when we color the set by at least three colors. We consider the $r$-commonness for $r\geq 3$ of linear equation with odd number of terms, and we also prove that any $2$-uncommon equation is $r$-uncommon for any $r\geq 3$.
 \end{abstract}


\maketitle


\section{Introduction}
\subsection{Background and Motivation}
In 1996, Graham, R\"{o}dl, and Ruci\'{n}ski \cite{GRR} asked the number of monochromatic solutions of Schur's equation $x+y=z$ with $(x,y,z)\in [n]^3$, where we do $2$ coloring of $[n]:=\{1,\dots,n\}$. Robertson and Zeilberger \cite{RZ} showed that the minimum number of monochromatic Schur triples in a $2$-coloring of $[n]$ is asymptotically $n^2/11+O(n)$. This is less than $(1/8+o(1))n^2$, which is the expected number of monochromatic solutions if we do uniformly random coloring.

One can ask the similar question for more general linear equation
\begin{align*}
    a_1x_1+\dots+a_kx_k=0,\quad a_1,\dots,a_k\in \Z\backslash\{0\}.
\end{align*}
This problem has been studied for several linear equations, including generalized Schur triples, $K$-term arithmetic progressions, and constellations. In fact we are still unable to know which linear equations having uniformly random $2$-coloring asymptotically minimize the number of monochromatic solutions. 

We define $k$-term linear equation is \textbf{$2$-common over the integers} if any $2$-coloring of $[n]$ has at least as many monochromatic solutions asymptotically (as $n\rightarrow \infty$) as a uniformly random coloring. Otherwise, we say the equation is \textbf{$2$-uncommon.} More generally if we change $2$ to any positive integer $r$ greater than $1$, we can define linear equations is $r$-common over the integers in similar fashion. 

Recently, Costello and Elvin \cite{CE} showed that all $3$-term equations are $2$-uncommon over the integers. In the same paper, they conjectured that an equation is $2$-common over the integers if and only if the number of terms is even and has canceling partition. We say the linear equation 
\begin{align}
\label{eq:terms}
a_1x_1+\dots+a_kx_k=0 
\end{align}
has \textbf{canceling partition} if we can partition the coefficients into pairs $\{a_i,a_j\}$ such that $a_i+a_j=0$. Clearly if canceling partition exists then $k$ must be even. 

More recently, Dong, Mani, Pham, and Tidor \cite{DMPT} showed that the conjecture is false by showing the linear equation $x_1+2x_2-x_3-2x_4=0$ is uncommon over the integers. The case when the number of terms in equation~\eqref{eq:terms} is odd $k>3$ still unknown. 

While commonness over the integers is still a mystery, Versteegen \cite{Ver} proved that an equation is $2$-uncommon over finite abelian group $A$, with order of $A$ is coprime to any coefficient of the equation if and only if $k$ is even and has no canceling partition. This generalizes the result from \cite{FPZ} who proved the same result over finite fields. See also \cite{SW} for introduction to the topic.

We also note that another motivation to consider $r$-commonness of a linear equation is because Sidorenko property of an equation, which is introduced first in \cite{SW}. The notion is inspired from Sidorenko's conjecture on graph \cite{Sid}, which is still a major open problem in extremal graph theory. We recall the definition here. Given a linear equation $L : a_1x_1+\dots+a_kx_k$ over a finite abelian group $G$, if $\mathcal{C}(L)$ denotes the number of solutions of $L=0$ in $G^k$, we call $L$ is \textbf{Sidorenko in $G$} if for every $A\subseteq G$ we have
\[
t_L(1_A)\geq \left( \frac{|A|}{|G|} \right)^k
\]
where $1_A$ is indicator function of $A$ and
\[
t_L(1_A):=\frac{1}{|\mathcal{C}(L)|} \sum_{v\in \mathcal{C}(L)}\prod_{i=1}^k 1_A(v_i).
\]
Clearly if an equation is Sidorenko over $G$, then it is $r$-common for every $r\geq 2$. The results for finite fields is known in \cite{FPZ} and for general finite abelian groups in \cite{Ver}. One can think the notion of $r$-commonness for $r\geq 3$ to be the case between Sidorenko and $2$-common, which is referred only as common in existing literatures.

\subsection{Main Results}
We consider now the phenomenon of $r$-commonness for linear equation with $r\geq 3$. We first consider $r$-commonness over the integers of $k$-term linear equations where $k>1$ is odd positive integer, for every $r>2$. The problem is easier compared to $2$-commonness. In fact, they are all $r$-uncommon over the integers as a corollary of the following theorem.
\begin{theorem}
    \label{thm:r uncommon odd}
    Let $G$ be arbitrary nontrivial finite abelian group and $ E:a_1x_1+\dots+a_{2m+1}x_{2m+1}=0$ be an arbitrary linear equation such that $a_i\in \Z\backslash\{0\}$ for each $i=1,\dots,2m+1, m\geq 1$ with $|G|$ is coprime with any $a_i,i=1,\dots,2m+1$. Then the equation $E$ is $r$-uncommon over $G$ for every $r\geq 3$.
\end{theorem}
\begin{cor}
\label{cor:r uncommon odd}
    Let $E:a_1x_1+\dots+a_{2m+1}x_{2m+1}=0, a_i\in \Z\backslash\{0\}$ for each $i=1,\dots,2m+1$ and $m\geq 1$ be a $2m+1$-linear equation. Then $E$ is $r$-uncommon over the integers for every $r\geq 3$.
\end{cor}

Next we consider when the number of terms in the equation is even. As mentioned, from \cite{Ver}, it is known that every equation that has no canceling partition is $2$-uncommon over abelian group $A$ provided the order of $A$ is coprime to every $a_i$. In particular, by Lemma \ref{lem: CostElv}, we have such equation is $2$-uncommon over the integers. We also already mentioned that $x_1+2x_2-x_3-2x_4=0$ is $2$-uncommon over the integers. Now we also have these equations are $r$-uncommon by using different method of proof with Corollary \ref{cor:r uncommon odd}.
\begin{theorem}
    \label{thm:r uncommon even}
    Let $E : a_1x_1+\dots+a_kx_k=0,a_i\in \Z\backslash\{0\},k\geq 3$ for each $i=1,\dots,k$ be $2$-uncommon equation over the integers. Then the equation is also $r$-uncommon over the integers for any $r\geq 3$.
\end{theorem}
\begin{cor}
    Every linear equation $a_1x_1+\dots+a_{2m}x_{2m}=0,m\geq 2$ that has no canceling partition is $r$-uncommon over the integers. The same is true for $x+2y-z-2w$.
\end{cor}

\section{Notation and Convention}
We cover several notations that we constantly use here. We denote $f=O(g)$ or $f\ll g$ or if there exists constant $C$ such that $|f|\leq Cg$. We denote $p$ for prime, and cyclic group of order $\ell$ by $\Z/\ell\Z$ for positive integer $\ell>1$. Let $f,g : G\rightarrow [0,1]$, which we associate with a probabilistic coloring via
\begin{align*}
    f(t)&=\mathbb{P}[t\text{ is the first color}]\\
    g(t)&=\mathbb{P}[t\text{ is the second color}].
\end{align*}
One may think that we use red, green, and blue as colors, and red is the first color, green is the second color, and blue is the third color. We define the \textbf{Fourier transform of} $f$, denoted by $\hat{f}$, by
\[
\hat{f}(\xi):=\frac{1}{|G|}\sum_{t\in G} f(t)\e(-\xi\cdot t)
\]
where $\e(x)=e^{2\pi ix}$ and $\xi$ is homomorphism from $\widehat{G}$ to $\R/\Z$ acting as $\xi : t\mapsto \xi\cdot t$. Here $\widehat{G}$ is the dual of $G$.
We define the Fourier transform of $g$, denoted by $\hat{g}$ similarly.

We can write the expected number of red solutions of $a_1x_1+\dots+a_{2m+1}x_{2m+1}$ over $G$ in terms of Fourier transforms:
\[
\mathbb{E}[\text{number of red solutions}]=|G|^{2m}\sum_{t\in G} \hat{f}(a_1t)\dots \hat{f}(a_{2m+1}t).
\]
Note that the above formula is valid only if at least one of $a_1,\dots,a_{2m+1}$ is coprime to $|G|$, which we always assume. The expected proportion of monochromatic solutions in $G$ is
\begin{align}
\begin{split}
\label{eq:sols}
&\mu_{a_1x_1+\dots+a_{2m+1}x_{2m+1}=0}(f,g)\\
&=\sum_{t\in G}\hat{f}(a_1t)\dots\hat{f}(a_{2m+1}t)
+\sum_{t\in G}\hat{g}(a_1t)\dots\hat{g}(a_{2m+1}t)\\
&\quad+\sum_{t\in G}(\widehat{1-f-g})(a_1t)\dots(\widehat{1-f-g})(a_{2m+1}t).
\end{split}
\end{align}
\section{Linear Equations That Are $r$-Uncommon}
We begin by proving Theorem \ref{thm:r uncommon odd}. Corollary \ref{cor:r uncommon odd} follows from Theorem \ref{thm:r uncommon odd} and the following lemma, which is \cite[Lemma 2.1]{CE}. For arbitrary set $S$ and arbitrary linear equation $E:a_1x_1+\dots+a_kx_k=0$, we denote the proportion of the minimum number of monochromatic solutions with the number of total solutions of $E$ in $S^k$ by $\mu_E(S)$.
\begin{lemma}
\label{lem: CostElv}
    Let $E : a_1x_1+\dots+a_{2m+1}x_{2m+1}=0$. Then we have
    \[
    \limsup_{n\rightarrow \infty} \mu_E([n])\leq \mu_E(\Z/\ell\Z)
    \]
    for any positive integer $\ell$.
\end{lemma}
Now we proceed to the proof of Theorem \ref{thm:r uncommon odd}.
\begin{proof}[Proof of Theorem \ref{thm:r uncommon odd}]
Let $E:a_1x_1+\dots+a_{2m+1}x_{2m+1}=0, a_i\in \Z\backslash\{0\}$. We prove that the equation is $3$-uncommon first, and then generalize to any $r>3$. We need to find $f$ and $g$ such that \eqref{eq:sols} is less than $\frac{1}{3^{2m}}$. But we have
\begin{align*}
    &\mu_{a_1x_1+\dots+a_{2m+1}x_{2m+1}=0}(f,g)\\
    &=\frac{1}{3^{2m}}
    +\sum_{t\in G\backslash\{0\}}\hat{f}(a_1t)\dots\hat{f}(a_{2m+1}t)+\sum_{t\in G\backslash\{0\}}\hat{g}(a_1t)\dots\hat{g}(a_{2m+1}t)\\
&\quad+\sum_{t\in G\backslash\{0\}}(\widehat{1-f-g})(a_1t)\dots(\widehat{1-f-g})(a_{2m+1}t)\\
&=\frac{1}{3^{2m}}+\sum_{t\in G\backslash\{0\}}\hat{f}(a_1t)\dots\hat{f}(a_{2m+1}t)
+\sum_{t\in G\backslash\{0\}}\hat{g}(a_1t)\dots\hat{g}(a_{2m+1}t)\\
&\quad+\sum_{t\in G\backslash\{0\}}(-\hat{f}-\hat{g})(a_1t)\dots(-\hat{f}-\hat{g})(a_{2m+1}t).
\end{align*}
The last equality follows from the fact for $s\neq 0$, we have $(\widehat{1-f})(s)=-\hat{f}(s)$.

We assume without loss of generality $\hat{f}(0)=\hat{g}(0)=\frac13$, which is equivalent to requiring overall red, green, and blue appear with equal probability. To get the proportion to be less than what we expect from uniformly random coloring, it is enough to find $f$ and $g$ such that $\hat{f}(0)=\hat{g}(0)=\frac13$ and
\begin{align*}
&\sum_{t\in G\backslash\{0\}}\hat{f}(a_1t)\dots\hat{f}(a_{2m+1}t)+\sum_{t\in G\backslash\{0\}}\hat{g}(a_1t)\dots\hat{g}(a_{2m+1}t)\\
&+\sum_{t\in G\backslash\{0\}}(-\hat{f}-\hat{g})(a_1t)\dots(-\hat{f}-\hat{g})(a_{2m+1}t)<0.
\end{align*}
We call the quantity of the above sum as \textbf{deviation}.

By Fourier inversion, $f$ and $g$ are uniquely determined by their Fourier coefficients. First, we note that $f$ and $g$ are real-valued if and only if $\hat{f}$ and $\hat{g}$ are Hermitian, i.e., $\overline{\hat{f}(s)}=\hat{f}(-s),\overline{\hat{g}(s)}=\hat{g}(-s)$, so we need to make sure this condition holds for $\hat{f}$ and $\hat{g}$. Second, we need to make sure that the ranges of $f$ and $g$ are subset of $[0,1]$. To do this, we use the \textbf{Fourier inversion formula}
\[
f(u)=\sum_{\xi\in \widehat{G}} \hat{f}(\xi)\e(\xi\cdot u), u\in G
\]
where we again use $\e(x)=e^{2\pi i x}$.
By triangle inequality, we have
\begin{align}
\label{ineq:f}
|f(u)-1/3|\leq \sum_{t\in G\backslash\{0\}}|\hat{f}(t)|.
\end{align}
By the same calculation, we also have
\begin{align}
\label{ineq:g}
|g(u)-1/3|\leq \sum_{t\in G\backslash\{0\}}|\hat{g}(t)|.
\end{align}
With these observations, we now construct $f$ and $g$ explicitly as follows. 

Define $f$ by having 
\[
\hat{f}(s)=-\frac{2}{p^2},\quad s\in G\backslash\{0\}
\]
and $g$ by having
\[
\hat{g}(s)=\frac{1}{p^2},\quad s\in G\backslash\{0\}
\]
We can take prime $p$ such that $p>|a_i|$ and $\gcd(a_i,p)=1$ for each $i=1,\dots,2m+1$ with
\begin{align*}
    0\leq \frac13-\frac{2(p-1)}{p^2}=\frac13-\sum_{t\in G\backslash\{0\}}\hat{f}(t)&\leq f(u)\leq \frac13+\frac{2(p-1)}{p^2}=\frac13+\sum_{t\in G\backslash\{0\}}\hat{f}(t)\leq 1\\
    0\leq \frac13-\frac{(p-1)}{p^2}=\frac13-\sum_{t\in G\backslash\{0\}}\hat{g}(t)&\leq g(u)\leq \frac13+\frac{(p-1)}{p^2}=\frac13+\sum_{t\in G\backslash\{0\}}\hat{g}(t)\leq 1.
\end{align*}
for any $u\in G$. The above inequalities follow from the inequalities \eqref{ineq:f} and \eqref{ineq:g} and the values of $\hat{f}, \hat{g}$. These give us images of $f$ and $g$ lie in $[0,1]$. We also remark $f$ and $g$ are real since $\hat{f}$ and $\hat{g}$ are Hermitian since $\hat{f}$ and $\hat{g}$ are real and both only take one value except at zero.

Another thing we need is $0<f(u)+g(u)<1$ for any $u\in G$. This can be guaranteed by taking $p$ large so that
\[
0<\frac23-\frac{3(p-1)}{p^2}<\frac23+\frac{3(p-1)}{p^2}<1.
\]

By the Fourier coefficients of $f$ and $g$ again, we have the deviation is
\begin{align}
\label{eq:dev}
-\left( \frac{2}{p^2} \right)^{2m+1}+\left( \frac{1}{p^2} \right)^{2m+1}+\left( \frac{1}{p^2} \right)^{2m+1}<0.
\end{align}
Therefore, every linear equation with $2m+1$ terms is $3$-uncommon.

To generalize to other $r>3$, we need to have $r-1$ functions $f_1,\dots,f_{r-1}$ such that the deviation is negative. As in the case when $r=3$, for other $r>3$ we require $\hat{f}_j(0)=\frac{1}{r}$ for every $j\in \{1,\dots,r-1\}$. Now we assign the values of their Fourier coefficients as follows. For $j=1,2$, we have
\begin{align*}
    \hat{f}_1(s)=-\frac{2}{p^2},&\quad s\in G\backslash\{0\},\\
    \hat{f}_2(s)=\frac{1}{p^2},&\quad s\in G\backslash\{0\}.
\end{align*}
where $p$ large in terms of $a_1,\dots,a_{2m+1}$, and also in terms of $r$ so that $0\leq f_j(u)\leq 1,u\in G,j=1,2$.

Then we define $\hat{f}_j(s)=0$ for any $j=3,\dots,r-1$ and any $s\in G\backslash\{0\}$. Now clearly we have that $\hat{f}_j$ is Hermitian and $f_j(u)\in [0,1]$ for any $u\in G$ and any $j=1,\dots,r-1$.
By doing similar calculations as in the case when $r=3$, we also get the deviation is negative, and we can choose $p$ such that $0<f_1(u)+\dots+f_{r-1}(u)<1$ for any $u\in G$. Thus this concludes the proof of Theorem \ref{thm:r uncommon odd}.
\end{proof}
\begin{remark}
    Note that if we choose a sufficiently large $p$, although the deviation in \eqref{eq:dev} is negative, it is negligible and the quantity is far smaller compare to the size of the group $G$. It would be interesting to obtain such $r$-uncommon coloring explicitly.
\end{remark}
To prove Theorem \ref{thm:r uncommon even}, we use probabilistic method instead of using Fourier method.
\begin{proof}[Proof of Theorem \ref{thm:r uncommon even}]
     We claim that if an equation is $r-1$-uncommon then it is $r$-uncommon for any $r\geq 3$, and this clearly implies the statement of the theorem immediately.
    
    Note that the total number of solutions with $x_i=x_j$ for some $i\neq j$ is $O(n^{k-1})$, so if $n$ is large, this kind of solution is negligible in our counting. So we can focus on solution with pairwise distinct coordinates. 
    
    Now let $(x_1,\dots,x_k)$ be solution of our linear equation in $[n]^k$ with $x_i\neq x_j$ for $i\neq j$ and $n$ is sufficiently large integer such that the equation $E$ is $2$-uncommon over $[n]$. We consider $r-1$-coloring of $[n]$ such that the equation is $r-1$-uncommon over $[n]$. We choose $\lfloor \frac{n}{r} \rfloor$ elements of $[n]$ uniformly at random, then we color those elements by the $r$-th color. Then if $(x_1,\dots,x_k)$ is monochromatic with previous coloring and $x_i\neq x_j$ for $i\neq j$, the probability that it will be monochromatic with the same color is
    \[
    \left( \frac{r-1}{r} \right)^k.
    \]
    If $(x_1,\dots,x_k)$ is arbitrary solution, it becomes monochromatic with the $r$-th color with probability
    \[
    \left( \frac{1}{r} \right)^k.
    \]

    Therefore, since the equation is $r-1$-uncommon, the expected proportion number of monochromatic solutions with $r$ colors and the number of total solutions is less than
    \[
    \left( \frac{r-1}{r} \right)^k\frac{1}{(r-1)^{k-1}}+\left( \frac{1}{r} \right)^k=\left( \frac{1}{r} \right)^{k-1}.
    \]
    Hence, there is an $r$-coloring of $[n]$ such that the equation $E$ is $r$-uncommon.

    Since we start with equation that is $2$-uncommon, we get the result.
\end{proof}
\section{Open Problem}
Costello and Elvin \cite{CE} used the Fourier method to show that the equation of type
\begin{align}
\label{eq:even}
x_1+\dots+x_m=x_{m+1}+\dots+x_{2m}
\end{align}
is $2$-common over the integers for any $m\geq 2$.  We showed that $2$-uncommonness implies $r$-uncommonness for any $r\geq 3$, but we still do not know whether if an equation is $2$-common then it is also $r$-common. Even in the simplest case where $m=2$ in \eqref{eq:even} we are not able yet to determine whether it is $3$-common or not. 

This motivates the following question.
\begin{question}
    Is there a linear equation that is $r$-uncommon for some $r\geq 3$ but it is $s$-common for some $s<r$?
\end{question}

We are particularly interested in equation that has canceling partition that is $2$-common over the integers since they are still poorly understood.

\section*{Acknowledgements}
The author would like to thank Fernando Xuancheng Shao for helpful discussions and suggestions during preparation of this work. The author is also indebted to anonymous referee for helpful suggestions and corrections.

\end{document}